\documentclass[11pt]{amsart}
\usepackage[a-1b]{pdfx}
\usepackage[text={6.1in,8.5in},centering]{geometry}
\usepackage{amssymb,amsmath,amsthm,mathtools,extpfeil}
\usepackage[all,cmtip]{xy}
\usepackage[shortlabels]{enumitem}
\usepackage{xcolor}

\newtheorem*{thm*}{Theorem}
\newtheorem*{lem*}{Lemma}
\newtheorem{thmA}{Theorem}

\newtheorem{lem}[equation]{Lemma}
\newtheorem{prop}[equation]{Proposition}
\newtheorem*{prop*}{Proposition}
\newtheorem{cor}[equation]{Corollary}

\theoremstyle{definition}

\newtheorem*{rmkdef}{Remark on the definition}

\numberwithin{equation}{section}

\DeclareMathOperator{\id}{id}
\DeclareMathOperator{\img}{im}

\DeclareMathOperator{\coker}{coker}
\DeclareMathOperator{\Dil}{Dil}

\DeclareMathOperator{\Lip}{Lip}
\DeclareMathOperator{\Hom}{Hom}
\newcommand{\ph}{\varphi}

\begin{document}
\title{Positive weights and self-maps}
\author[F.~Manin]{Fedor Manin}
\address{Department of Mathematics, University of California, Santa Barbara, CA, United States}
\email{manin@math.ucsb.edu}
\begin{abstract}
  Spaces with positive weights are those whose rational homotopy type admits a
  large family of ``rescaling'' automorphisms.  We show that finite complexes
  with positive weights have many genuine self-maps.  We also fix the proofs of
  some previous related results.
\end{abstract}
\maketitle

\section{Main result}
Following \cite{BMSS}, who attribute the term to Morgan and Sullivan, we say that
a simply connected space has \emph{positive weights} if its rational homotopy
type has a one-parameter family of ``rescaling'' automorphisms.  A given space
will often have many such families.  A precise definition is given in
\S\ref{S:defs}.

The main result of this paper is that of any such family consisting of a
$\mathbb{Q}$'s-worth of rational automorphisms, a $\mathbb{Z}$'s-worth of them
can be realized as self-maps of any finite complex of that homotopy type.
\begin{thmA} \label{thm:appx}
  Let $Y$ be a finite simply connected CW complex with positive weights, as
  witnessed by a one-parameter family of homomorphisms
  $\lambda_t:Y_{(0)} \to Y_{(0)}$.  Let $\ell:Y \to Y_{(0)}$ be the rationalization
  map.  Then there is an integer $t_0 \geq 1$ such that for every
  $z \in \mathbb{Z}$, there is a genuine map $f_z:Y \to Y$ whose rationalization
  is $\lambda_{zt_0}$, that is, such that
  $\ell \circ f_z \simeq \lambda_{zt_0} \circ \ell$.
\end{thmA}
The class of spaces with positive weights is large; for example, it includes all
formal spaces \cite[Thm.~12.7]{SulLong}, homogeneous spaces
\cite[Prop.~3.7]{BMSS}, and smooth complex algebraic varieties \cite{Morgan}.
Indeed, it is somewhat nontrivial to find a simply connected space which does not
have positive weights.  The lowest-dimensional nonexample, as far as we know, is
a complex given in \cite[\S4]{MT} which is constructed by attaching a 12-cell to
$S^3 \vee \mathbb{C}\mathbf{P}^2$; other, much higher-dimensional non-examples
are given in \cite{ArLu,CLoh,CV,Am}.

We state a corollary for formal spaces which follows immediately by
\cite[Theorem 12.7]{SulLong}, which states that every formal rational homotopy
type has a one-parameter family of automorphisms which induces the grading
automorphisms on $H^*(Y;\mathbb Q)$ which send a class
$\alpha \in H^n(Y;\mathbb{Q})$ to $t^n\alpha$.
\begin{cor} \label{formal}
  Let $Y$ be a simply connected formal finite complex.  Then there is an integer
  $t_0 \geq 1$ such that for every $z \in \mathbb Z$, there is a map
  $f_z:Y \to Y$ which induces multiplication by $(zt_0)^n$ on $H^n(Y;\mathbb Q)$
  for every $n$.
\end{cor}

While this paper is motivated by an application of this corollary to quantitative
homotopy theory, the author hopes that it will be of wider interest.

\section{Prior work}
The statement of Theorem \ref{thm:appx} is not quite present in the literature,
although a number of prior papers state similar results and give arguments which
would imply this theorem.  However, the author was unable to fill in the details
of these arguments; this is the major motivation for this short paper.
\begin{itemize}
\item A slight weakening of Corollary \ref{formal} was originally stated by Shiga
  \cite{Shiga}.  However, his proof has significant issues.  In particular,
  the argument on the bottom of p.~432 seems to rely on the claim that, e.g., in
  the equation $\mathbb Q^2=\mathbb Q(1,1) \oplus \mathbb Q(-1,1)$ one can
  replace $\mathbb Q$ with $\mathbb Z$, and the author was not able to fix the
  argument to avoid this.
\item A number of similar results are discussed in \cite{BMSS}.  The main result
  of that paper is that for finite complexes, the positive weight condition is
  equivalent to \emph{$p$-universality} for any $p$ a prime or zero.  A space is
  $p$-universal, a notion introduced by Mimura and Toda \cite{MT}, if for every
  $q \neq p$ it has a self-map that induces isomorphisms on mod-$p$ homology
  (rational in the case $p=0$) and the zero map on mod-$q$ homology.  In
  particular, $0$-universality is closely related to the conclusion of Theorem
  \ref{thm:appx}, and Proposition 3.3 and Lemma 3.4 of \cite{BMSS} are similar to
  Lemmas \ref{lem:KtoK} and \ref{lem:Kn} below.  However, in this case the author
  was again unable to complete the argument as written: the map $f_q$ constructed
  in the inductive step of Lemma 3.4 depends on a choice of homotopy and it's
  unclear how or whether one can pick a version that would satisfy the claimed
  conditions.  In this paper, we give an alternate proof of Proposition 3.3 of
  \cite{BMSS}.
\item Amann \cite[Theorem 4.2]{Am} asserts a result similar to Shiga's, but for
  spaces with positive weights in general.  However, the proof again contains a
  mistake: an obstruction lies in the cohomology of the wrong space.  Amann has
  pointed out to the author that this mistake is similar to that in the published
  proof of \cite[Lemma B.1]{BK}, which has been fixed on the arXiv, and can be
  fixed in a similar way.  This is different from our method, but could be used
  to give a slightly weaker form of Theorem \ref{thm:appx}.
\end{itemize}
Our proof method has major similarities to those of \cite{BMSS} (in overall
strategy) and \cite{Am} (in the use of the Moore--Postnikov tower of the
rationalization map $Y \to Y_{(0)}$).

\section{Positive weights} \label{S:defs}

We assume knowledge of Sullivan's model of rational homotopy theory; the reader
is referred to \cite{FHT} or \cite{GrMo} for the basics.

Let $Y$ be a simply connected space of finite type, and denote its Sullivan
minimal DGA by $\mathcal M_Y^*$.  Then $Y$ has \emph{positive weights} if there
is a set $\{x_i\}$ of indecomposable generators of $\mathcal M_Y^*$ and
corresponding integers $n_i \geq 1$ such that for each $t \in \mathbb{Q}$, there
is a homomorphism $\lambda_t:\mathcal{M}_Y^* \to \mathcal{M}_Y^*$ such that
$\lambda_t(x_i)=t^{n_i}x_i$.

Notice that when $t \neq 0$, this $\lambda_t$ is an automorphism; the set
$\{\lambda_t:t \in \mathbb Q^\times\}$ is a subgroup of the automorphism group of
$\mathcal{M}_Y^*$ and is called a \emph{one-parameter subgroup} or \emph{family}.

Since there is an equivalence of homotopy categories between rational
spaces\footnote{That is, simply connected CW complexes whose homotopy groups are
  rational vector spaces.} of finite type and their minimal DGAs, such an
automorphism $\lambda_t$ induces a homotopy automorphism of the rationalization
$Y_{(0)}$, which by an abuse of notation we may also call $\lambda_t$.

Note that there are often many possible choices of basis and of the $n_i$.  For
example, given one such family $\lambda_t$ any other automorphism $\ph$ of
$\mathcal M_Y^*$, one can get a new family by conjugating $\lambda_t$ by $\ph$.
Concretely, let $Y=S^2 \times S^3$, and choose:
\begin{itemize}
\item $\lambda_t$ to be the product of degree $t$ maps on $S^2_{(0)}$ and
  $S^3_{(0)}$;
\item $\ph$ to be the rationalization of the map
  \[S^2 \times S^3 \to S^2 \times S^3\]
  which sends $S^2$ to itself and $S^3$ to $S^2 \vee S^3$ via
  $\operatorname{Hopf}+\id_{S^3}$.\footnote{Such a map exists because the
    Whitehead product $[\id_{S^2},\operatorname{Hopf}+\id_{S^3}]$ is zero in
    $S^2 \times S^3$.}
\end{itemize}
Then $\ph^{-1}\lambda_t\ph$ and $\lambda_t$ are different families of
automorphisms.

It is clear from the definition that $\lambda_t$ induces diagonalizable
automorphisms on $\pi_n(Y) \otimes \mathbb{Q}$.  The same is true for homology
and cohomology:
\begin{prop} \label{prop:diagH}
  If $\lambda_t:\mathcal M_Y^* \to \mathcal M_Y^*$ is a one-parameter family of
  automorphisms, then there are also bases for $H^*(Y;\mathbb{Q})$ and
  $H_*(Y;\mathbb{Q})$ consisting of eigenvectors of the maps induced by
  $\lambda_t$.
\end{prop}
\begin{proof}
  The action of $\lambda_t$ on $\mathcal M_Y^*$ is diagonalizable.  Since
  $\lambda_t$ sends cocycles in $\mathcal M_Y^*$ to cocycles, they form an
  invariant subspace, which is therefore also diagonalizable.  This
  diagonalization passes to the quotient by coboundaries, giving the result for
  cohomology.  Dualizing gives us the same result for homology.
\end{proof}

In \cite[Theorem 2.7]{BMSS}, it is shown that the positive weight condition is
independent of coefficients: a minimal $\mathbb Q$-DGA has positive weights if
and only if its tensor product with $\mathbb R$ or another larger field does.
Many additional topological and algebraic properties of the positive weight
condition are discussed in \cite{PWHT}, including closure under operations such
as wedge and product.  Most interestingly, the condition is its own
Eckmann--Hilton dual.

\begin{rmkdef}
  In the definition of a space with positive weights, the assignment of
  ``weights'', $x_i \mapsto n_i$, extends uniquely to a second grading on
  $\mathcal M_Y^*$ that respects the multiplication.  Then a space with positive
  weights is one which has such a second grading with respect to which the
  differential has degree zero.  This obviously equivalent definition is the one
  more often given, e.g.\ in \cite[Definition 2.1]{BMSS}.

  In \cite[Proposition 2.3]{BMSS}, this equivalence is shown over other
  coefficient fields.
\end{rmkdef}

\section{Corollaries and related results}

A useful result closely related to Theorem \ref{thm:appx} shows that there are
many maps between two spaces of the same positive-weight rational homotopy type:
\begin{thmA} \label{thm:XtoY}
  Let $Y$ and $Y'$ be two rationally equivalent simply connected finite
  complexes, with rationalizations $\ell:Y \to Y_{(0)}$ and
  $\ell':Y' \to Y_{(0)}$.  Let $\lambda_t:Y_{(0)} \to Y_{(0)}$ be a one-parameter
  family of homotopy automorphisms.  Then there are maps $f:Y \to Y'$ and
  $g:Y' \to Y$ and a $t \in \mathbb Z$ such that $\ell gf \simeq \lambda_t\ell$
  and $\ell'fg \simeq \lambda_t\ell'$.
\end{thmA}
We prove this along with Theorem \ref{thm:appx} in the next section.

A manifold is \emph{flexible} in the sense of Crowley and L\"oh \cite{CLoh} if it
has self-maps of infinitely many degrees (or equivalently, at least one degree
other than $0$ and $\pm 1$).  An immediate corollary of Theorem \ref{thm:appx}
and Proposition \ref{prop:diagH} is the following result:
\begin{cor}
  Manifolds with positive weights are flexible.
\end{cor}
This was previously essentially stated by Amann \cite[Theorem 4.2]{Am}.  Another,
quicker proof is implicit in a recent paper of Costoya, Mu\~noz, and Viruel
\cite[Theorem 3.2]{CMV}.

Finally, we explore simple quantitative implications of our results.  Given
finite complexes $X$ and $Y$ with a piecewise Riemannian metric, the
\emph{growth function} $g_{[X,Y]}(L)$ of the set $[X,Y]$ of homotopy classes of
maps $X \to Y$ is the number of classes that have representatives of with
Lipschitz constant at most $L$, as a function of $L$.  This notion was first
studied by Gromov \cite{GrHED,GrMS,GrQHT}.  While the definition uses
the metrics on $X$ and $Y$, the asymptotics of this function depend only on the
homotopy types of the two spaces.  Indeed, in \cite[\S6]{IRMC} it was shown,
based on the results of \cite{BMSS}, that if $X$ and $Y$ have positive weights,
then the growth function only depends on their rational homotopy type.

In fact, $g_{[X,Y]}(L)$ is always bounded by a polynomial in $L$ when $Y$ is
simply connected or, more generally, nilpotent \cite[Corollary 4.7]{IRMC}.  On
the other hand, since $[X,Y]$ is more or less the set of solutions to a system of
diophantine equations, general lower bounds are hard to come by.  However, for
spaces with positive weights, Theorem \ref{thm:appx} provides such a lower bound:
\begin{thmA}
  Suppose that $Y$ is a simply connected finite complex with positive weights.
  Then the growth function $g_{[Y,Y]}(L)$ is bounded below by $L^r$ for some
  rational $r$.
\end{thmA}
\begin{proof}
  By Theorem \ref{thm:appx} there is a sequence of maps $f_z:Y \to Y$ realizing
  $\lambda_{zt_0}:\mathcal M_Y^* \to \mathcal M_Y^*$ for every $z \in \mathbb{Z}$.
  The latter induce maps on the $\mathbb R$-minimal DGA of $Y$ which we likewise
  call $\lambda_{zt_0}$.  Let $m_Y:\mathcal M_Y^*(\mathbb R) \to \Omega^*Y$ be a
  minimal model for the differential forms on $Y$.  By the shadowing principle
  \cite[Theorem 4--1]{PCDF}, we can find a map homotopic to $f_z$ with Lipschitz
  constant controlled by a notion of ``size'' of the homomorphism
  $m_Y\lambda_{zt_0}:\mathcal M_Y^*(\mathbb R) \to \Omega^*Y$.  Specifically, put a
  norm on the vector space $V_k=\Hom(\pi_k(Y),\mathbb R)$ of indecomposables in
  $\mathcal M_Y^*(\mathbb R)$ for each $k \leq \dim Y$, and for every
  $\ph:\mathcal M_Y^*(\mathbb R) \to \Omega^*Y$ let
  \[\Dil(\ph)=
  \max_{k \in \{2,\ldots,\dim Y\}} \lVert \ph|_{V_k} \rVert_{\mathrm{op}}^{1/k}.\]
  This measurement depends on the choices of which elements we consider
  indecomposable, of norms and of $m_Y$, but only up to a multiplicative
  constant.  In particular,
  \[\Dil(m_Y\lambda_{zt_0}) \leq \max \bigl\{C(Y)(zt_0)^{n_i/\dim x_i} \mid \deg(x_i) \leq \dim Y\bigr\}.\]
  By the shadowing principle, this means that we can choose $f_z$ so that
  \[\Lip f_z \leq C'(Y)\bigl[\max \bigl\{(zt_0)^{n_i/\dim x_i} \mid \deg(x_i) \leq \dim Y\bigr\}+1\bigr],\]
  and so $g_{[Y,Y]}(L) \geq L^{\min \{\dim x_i/n_i \mid \deg(x_i) \leq \dim Y\}}$.
\end{proof}

\section{Proof of Theorems \ref{thm:appx} and \ref{thm:XtoY}}
In this section, let $Y$ be a simply connected finite complex equipped with a
rationalization map $\ell:Y \to Y_{(0)}$ and a one-parameter family of
automorphisms $\lambda_t:\mathcal{M}_Y^* \to \mathcal{M}_Y^*$ which induce maps
$Y_{(0)} \to Y_{(0)}$ which we also call $\lambda_t$.

We prove Theorems \ref{thm:appx} and \ref{thm:XtoY} using a series of lemmas.
\begin{lem} \label{lem:Kn}
  For every $n$, there is a complex $K_n$ and a rational equivalence
  $q_n:K_n \to Y_{(0)}$ with the following properties:
  \begin{enumerate}[(i)]
  \item For $m \leq n$, $\pi_m(K_n)$ is free abelian.
  \item For each prime $p$, there is a map $r_{p,n}:K_n \to K_n$ such that
    $q_n \circ r_{p,n} \simeq \lambda_p \circ q_n$.  Moreover, the induced map on
    $\pi_m(K_n)$, $m \leq n$, has a $\mathbb{Z}$-eigenbasis.
  \item For $m>n$, $\pi_m(K_n)$ is a $\mathbb Q$-vector space, and therefore
    $q_{n*}:\pi_m(K_n) \to \pi_m(Y_{(0)})$ is an isomorphism.
  \end{enumerate}
\end{lem}
\begin{proof}
  We will construct the $K_n$ as successive stages of a Moore--Postnikov tower
  with base $Y_{(0)}$.  That is, we take $K_1=Y_{(0)}$ (in which case the base case
  is trivially true) and then construct a tower
  \[\xymatrix{
    K_{n+1} \ar[rd]^{q_{n+1}} \ar[d]^{f_{n+1}} \\
    K_n \ar[r]^{q_n} & Y_{(0)} 
  }\]
  such that the homotopy fiber of $f_{n+1}$ is a $K(\pi,n)$ (note the nonstandard
  indexing).  In particular, any such construction automatically satisfies (iii).

  Now suppose we have constructed $K_n$ and maps $q_n:K_n \to Y_{(0)}$ and
  $r_{p,n}:K_n \to K_n$ satisfying (i)--(iii).  We will now construct the next
  stage of the Moore--Postnikov tower, as well as maps $r_{p,n+1}$ which are lifts
  of $r_{p,n}$ along $f_{n+1}$, in the sense that
  \[f_{n+1} \circ r_{p,n+1} \simeq r_{p,n} \circ f_n.\]
  Since $\pi_n(K_n)$ is free abelian and we would like the same for
  $\pi_{n+1}(K_{n+1})$, we get
  \[\pi_{n+1}(K_n,K_{n+1}) \cong \mathbb{Q}^d/\mathbb{Z}^d,\]
  where $d$ is the rank of $\pi_{n+1}(Y)$.  Therefore, to fix $K_{n+1}$,
  it suffices to specify a $k$-invariant
  $\kappa \in H^{n+1}(K_n;(\mathbb{Q}/\mathbb{Z})^d)$ for the pullback diagram
  \[\xymatrix{
    K_{n+1} \ar[r] \ar@{->>}[d]_{f_{n+1}} &
    \mathcal PK((\mathbb Q/\mathbb Z)^d, n+1) \ar[d] \\
    K_n \ar[r]^-\kappa & K((\mathbb Q/\mathbb Z)^d, n+1).
  }\]
  
  Since $\mathbb{Q}/\mathbb{Z}$ is an injective $\mathbb Z$-module,
  \[H^{n+1}(K_n;\mathbb{Q}/\mathbb{Z}) \cong
  \Hom(H_{n+1}(K_n),\mathbb Q/\mathbb Z).\]
  Therefore we can think of $\kappa$ as a homomorphism
  $H_{n+1}(K_n) \to (\mathbb{Q}/\mathbb Z)^d$.  Moreover, the composition
  $\kappa \circ h:\pi_{n+1}(K_n) \to (\mathbb Q/\mathbb Z)^d$, where $h$ is the
  Hurewicz homomorphism, fits into the long exact sequence of homotopy groups
  \[\cdots \to \pi_{n+1}(K_n) \xrightarrow{\kappa \circ h}
  (\mathbb{Q}/\mathbb Z)^d \to \pi_n(K_{n+1}) \to \pi_n(K_n) \to 0 \to \cdots.\]
  Since we would like $\pi_n(K_{n+1}) \cong \pi_n(K_n)$, $\kappa \circ h$ needs to
  be surjective.
  
  Denote the $n$th Postnikov stage of $K_n$ by $(K_n)_n$.  To compute
  $H_{n+1}(K_n)$, we apply the Serre spectral sequence to the map
  $K_n \to (K_n)_n$, whose homotopy fiber is an $n$-connected rational space $W$
  with $H_{n+1}(W) \cong \mathbb Q^d$.  This gives us a short exact sequence
  \[0 \to A \to H_{n+1}(K_n) \to H_{n+1}((K_n)_n) \to 0\]
  where $A=\coker\bigl(d:H_{n+2}((K_n)_n) \to H_{n+1}(W)\bigr)$. Since $(K_n)_n$ is
  of finite type, the first term is $\mathbb{Q}^d$ modulo a finitely generated
  subgroup, and the last term is finitely generated.  In particular, the first
  term is an injective $\mathbb Z$-module, so the sequence splits.
  
  Now in order to pick the desired $\kappa$, we would like to understand the
  action of the various $r_{p,n}$ on $H_{n+1}(K_n)$.  Specifically, we would like
  to pick $\kappa$ so that $\img(\kappa \circ r_{p,n*}) \subseteq \img(\kappa)$.
  Then since $\kappa \circ f_{n+1*}=0$ by construction, we also get
  $\kappa \circ r_{p,n*} \circ f_{n+1*}=0$, and therefore there is a lift
  \[\xymatrix{
    K_{n+1} \ar@{-->}[r]^{r_{p,n+1}} \ar@{->>}[d]_{f_{n+1}} &
    K_{n+1} \ar[r] \ar@{->>}[d]_{f_{n+1}} &
    \mathcal PK((\mathbb Q/\mathbb Z)^d, n+1) \ar[d] \\
    K_n \ar[r]^{r_{p,n}} & K_n \ar[r]^-\kappa & K((\mathbb Q/\mathbb Z)^d, n+1).
  }\]
  
  The automorphism $\lambda_t$ induces diagonalizable linear transformations with
  eigenvalues $t^\alpha$ for various integers $\alpha$ on both
  $\pi_*(K_n) \otimes \mathbb Q$ and $H_*(K_n;\mathbb{Q})$.  In particular, we
  can choose a basis of eigenvectors $x_i$ for $\mathbb{Q}^d$, as well as
  additional eigenvectors $y_j \in H_{n+1}(K_n)$ which, together with those $x_i$
  whose images in $A \otimes \mathbb{Q}$ are nonzero, form a basis for
  $H_{n+1}(K_n;\mathbb{Q})$.  The $y_j$, together with a choice of splitting for
  the torsion elements, determine a splitting
  $s:H_{n+1}((K_n)_n) \to H_{n+1}(K_n)$.  The eigendecomposition determines the
  action of $r_{p,n}$ on homology except for its action on the torsion subgroup
  $B \subseteq H_{n+1}((K_n)_n)$.  However, the image of $B$ will always be
  contained in the finite subgroup $B \oplus \{a \in A: \lvert B \rvert a=0\}$.
  
  Now we fix $\kappa$.  Write the codomain as
  $\bigoplus_{i=1}^d (\mathbb Q/\mathbb Z)e_i$.  Then we set $\kappa \circ s=0$
  and, for $q \in \mathbb Q$, $\kappa(qx_i)=Nqe_i$, where $N$ is large enough
  that $\{a \in A: \lvert B \rvert a=0\}$ is sent to zero.  As a result,
  \[\kappa \circ r_{p,n*}(s(H_{n+1}((K_n)_n)))=0\quad\text{and}\quad
  \kappa \circ r_{p,n*}(x_i)=p^{n_i}x_i,\]
  and therefore $\img(\kappa \circ r_{p,n*}) \subseteq \img(\kappa)$.  This shows
  that $r_{p,n} \circ f_{n+1}$ lifts to a map $r_{p,n+1}$.  Moreover, the generators
  $x_i/N$ of
  \[\pi_{n+1}(K_{n+1})=\ker(\kappa \circ h)=\textstyle{\bigoplus_i \bigl\langle\frac{1}{N}x_i\bigr\rangle}\]
  form a $\mathbb{Z}$-eigenbasis for
  $(r_{p,n+1})_*:\pi_{n+1}(K_{n+1}) \to \pi_{n+1}(K_{n+1})$.
\end{proof}
\begin{lem} \label{lem:Z/pZ}
  Given $K_n$ and $r_{p,n}$ as in Lemma \ref{lem:Kn}, there is a power of $r_{p,n}$
  which induces the zero map on $H_*(K_n;\mathbb{Z}/p\mathbb{Z})$ for all
  $* \leq n$.
\end{lem}
\begin{proof}
  This is essentially the direction (b$'$) $\Rightarrow$ (b) of
  \cite[Theorem 2.1]{MOT}.
\end{proof}
\begin{lem} \label{lem:KtoK}
  There is a finite complex $K$ and a rational equivalence $q:K \to Y_{(0)}$ with
  the following properties:
  \begin{enumerate}[(i)]
  \item For each $m \leq \dim Y$, $\pi_m(K)$ is free abelian, and for each
    $m>\dim Y$, $H_m(K)=0$.
  \item For each prime $p$, there is a map $r_p:K \to K$ such that
    $q \circ r_p \simeq \lambda_p \circ q$.  Moreover, for every prime
    $p' \neq p$, $r_p$ is a $p'$-equivalence, i.e.~it induces isomorphisms on
    $H^*(K;\mathbb Z/p'\mathbb Z)$.
  \item For each prime $p$, there is a power $s_p$ of $r_p$ which induces the
    zero map on $H_*(K;\mathbb{Z}/p\mathbb{Z})$.
  \end{enumerate}
\end{lem}
\begin{proof}
  Let $n$ be the dimension of $Y$, and let $K_n$ be as in Lemma \ref{lem:Kn}.
  Since $(K_n)_n$ is of finite type, we can build a finite $n$-complex $K'$ with
  a map $\iota':K' \to K_n$ which induces isomorphisms on $H_m$ for every $m<n$
  and a surjection on $H_n$.  In particular, there is no obstruction to
  homotoping the map $r_{p,n} \circ \iota':K' \to K_n$ so that its image lands
  $\iota'(K')$; we can then lift this to a map $r_p':K' \to K'$.

  Now, by the Hurewicz theorem, since the homotopy fiber of $\iota'$ is
  $(n-1)$-connected, the Hurewicz map $\pi_{n+1}(K_n,K') \to H_{n+1}(K_n,K')$ is
  surjective.  Therefore we can add $(n+1)$-cells to $K'$ which kill the kernel
  of $\iota'_*:H_n(K') \to H_n(K_n)$.  Moreover, since $K'$ is $n$-dimensional,
  $H_n(K')$ is free abelian, and so is this kernel; thus we can do this without
  adding any homology in degree $n+1$.  We call the resulting $(n+1)$-complex
  $K$; by construction it satisfies (i).

  Note that $H^{n+1}(K;\pi)$ is zero integrally and rationally, but may be
  nontrivial with torsion coefficients.  In particular, $\iota'$ extends uniquely
  over the $(n+1)$-cells of $K$, since $\pi_{n+1}(K_n)$ is a $\mathbb{Q}$-vector
  space.  This gives us a map $\iota:K \to K_n$, and we then set
  $q=q_n \circ \iota$. This is a rational equivalence since $\iota$ induces an
  isomorphism on $H_m$ for $m \leq n$ and both $K$ and $Y_{(0)}$ are homologically
  trivial in degrees $>n$.  Finally, $r_p'$ extends to a map $r_p:K \to K$, and
  this extension is unique up to torsion; therefore
  $q \circ r_p \simeq \lambda_p \circ q$.

  Since the homotopy groups of $K_n$ are either rational or free abelian, and the
  maps induced by $r_{p,n}$ on the free abelian groups $\pi_m(K_n)$, $m \leq n$,
  have an eigenbasis whose vectors are multiplied by powers of $p$, $r_{p,n}$
  induces isomorphisms on $\pi_*(K_n) \otimes \mathbb Z/p'\mathbb Z$ for every
  prime $p' \neq p$.  By the mod $p'$ Hurewicz theorem, $r_{p,n}$ also induces
  isomorphisms on $H_*(K_n;\mathbb Z/p'\mathbb Z)$.  Now, $\iota'$ induces
  isomorphisms on mod $p'$ homology in degrees $m<n$ and a surjection in degree
  $n$, and factors into the inclusion $K \hookrightarrow K_n$ and the homology
  isomorphism $\iota$.  Since $r_p'$ induces isomorphisms in degrees $m<n$ and on
  the quotient of the surjection in degree $n$, $r_p$ induces isomorphisms on all
  mod $p'$ homology and cohomology groups.

  Condition (iii) comes directly from Lemma \ref{lem:Z/pZ}, since $r_p$ induces
  the same map on $H_{\leq n}$ as $r_{p,n}$, and $H_{>n}(K)=0$.
\end{proof}
\begin{lem} \label{lem:YtoK}
  There is a map $f:Y \to K$ which commutes with $\lambda_t$ after
  rationalization; more precisely, $q \circ f \simeq \lambda_t \circ \ell$ for
  some $t$.
\end{lem}
\begin{proof}
  Let $Z$ be an infinite telescope of mapping cylinders build using copies of $K$
  and maps
  \[r_2,r_2,r_3,r_2,r_3,r_5,\ldots\]
  Now consider the map $\hat q:Z \to Y_{(0)}$ extending $q$ on the first copy of
  $K$.  Let's call the inclusion map of this copy $i_1$, so then
  $\hat q \circ i_1 \simeq q$.

  By Lemma \ref{lem:KtoK}(i) and (ii), $\hat q$ induces isomorphisms on $\pi_m$
  for $m \leq n$.  Thus, by the Hurewicz theorem, $\hat q$ also induces
  isomorphisms on $H_m$ for $m \leq n$.  On the other hand, for $m>n$,
  \[H_m(Z) \cong H_m(Y_{(0)}) \cong 0.\]
  Therefore $\hat q$ is a homotopy equivalence, and so there is a map
  $\hat\ell:Y \to Z$ such that $\hat q \circ \hat\ell \simeq \ell$.  Since $Y$ is
  compact, this map lands in a finite set of mapping cylinders, and therefore we
  can homotope it into a single copy of $K$.  Let's call the inclusion map of
  this copy $i_t$, where $t=p_1 \cdots p_N$ such that
  \[i_1 \simeq i_t \circ R_t=i_t \circ r_{p_N} \circ \cdots \circ r_{p_1}.\]
  The resulting map is $f$.  To see that $q \circ f \simeq \lambda_t \circ \ell$,
  consider the diagram
  \[\xymatrixrowsep{1pc}\xymatrixcolsep{3pc}
  \xymatrix{
    Y \ar[r]_{\hat\ell} \ar@/^1pc/[rr]^\ell \ar[dd]_f & Z \ar[r]_-{\hat q} &
    Y_{(0)} \ar[dd]^{\lambda_t} \\
    & K \ar[u]_{i_1} \ar[ld]^{R_t} \ar[ru]_q \\
    K \ar[ruu]^{i_t} \ar[rr]_q && Y_{(0)}.
  }\]
  The triangles all commute up to homotopy by construction, and the bottom right
  square by Lemma \ref{lem:KtoK}(ii).  Moreover, every map in the diagram is a
  rational equivalence; in other words, after rationalization, every arrow is
  reversible.  This implies that the outer square commutes rationally.  But since
  the target $Y_{(0)}$ is a rational space, it commutes integrally as well.
\end{proof}
\begin{lem} \label{lem:KtoY}
  There is a map $g:K \to Y$ such that $g \circ f$ realizes $\lambda_{t_0}$ for
  some integer $t_0$, i.e.\ $\ell\circ g \circ f \simeq \lambda_{t_0} \circ \ell$.
\end{lem}
\begin{proof}
  We again use the proof of \cite[Theorem 2.1]{MOT}.  That theorem asserts the
  equivalence of several conditions for a finite simply connected CW complex $K$,
  including:
  \begin{enumerate}
  \item[(b)] For any prime $p$, there is a map $s_p:K \to K$ which induces the
    zero map on $H^*(K; \mathbb{Z}/p\mathbb{Z})$.
  \item[(a)] Given a rational equivalence $f:Y \to X$ between two CW complexes
    and a map $h:K \to X$, there are maps $g:K \to Y$ and $k:K \to K$ completing
    the diagram
    \[\xymatrix{
      K \ar@{-->}[r]^k \ar@{-->}[d]^g & K \ar[d]_h \\
      Y \ar[r]^f & X .
    }\]
  \end{enumerate}
  We showed in Lemma \ref{lem:KtoK} that $K$ satisfies (b) and that we can take
  $s_p$ to be a power of $r_p$.  The proof that (b) implies (a) goes through
  several steps, but the resulting map $k:K \to K$ is always a composition of
  $s_p$ for various $p$.  Applying (a) with $X=K$, $h=\id$, and $f$ the map from
  Lemma \ref{lem:YtoK}, we get a map $g:K \to Y$ such that $f \circ g:K \to K$ is
  a composition of various $r_p$, whose product is, let's say, $t_0$.  Then
  \[\ell gf \simeq \lambda_{t^{-1}}qfgf \simeq \lambda_{t^{-1}}\lambda_{t_0}qf \simeq
  \lambda_{t^{-1}}\lambda_{t_0}\lambda_t\ell \simeq \lambda_{t_0}\ell. \qedhere\]
\end{proof}
\begin{proof}[Proof of Theorem \ref{thm:appx}]
  For any $z \in \mathbb{Z}$, let $r_z:K \to K$ be the composition of the $r_p$'s
  in its prime decomposition.  Then $g \circ r_z \circ f$ is a map realizing the
  automorphism $\lambda_{zt_0}$.
\end{proof}
\begin{proof}[Proof of Theorem \ref{thm:XtoY}]
  Using Lemmas \ref{lem:YtoK} and \ref{lem:KtoY}, we construct maps
  \begin{align*}
    Y &\xrightarrow{f} K \xrightarrow{g} Y \\
    Y' &\xrightarrow{f'} K \xrightarrow{g'} Y'
  \end{align*}
  such that $\ell gf \simeq \lambda_t\ell$ and
  $\ell' g'f' \simeq \lambda_{t'}\ell'$.  Then $g'f:Y \to Y'$ and $fg':Y' \to Y$
  are the desired maps.
\end{proof}
Finally, the fact that the $r_p$ are $p'$-equivalences for every $p \neq p'$
completes an alternate proof of Proposition 3.3 in \cite{BMSS}.  This can be used
to recover the theorem that spaces with positive weights are $p$-universal for
every $p$.

\subsection*{Acknowledgements}
I would like to thank Manuel Amann, Richard Hain, and Antonio Viruel for
supplying missing references and other useful comments.  I would also like to
thank the very thorough and careful anonymous referee who graciously corrected
the inaccuracies that inevitably cropped up in the hastily written first version.
I was partly supported by NSF individual grant DMS-2001042.

\bibliographystyle{amsplain}
\bibliography{liphom}

\end{document}